\documentclass{article}
\usepackage{graphicx} 
\usepackage{amsmath}
\usepackage{amsthm}
\usepackage{amssymb} 
\usepackage[a4paper, margin=1in]{geometry}
\newtheorem{theorem}{Theorem}
\title{Some generalized inequalities in Riemannian Geometry}
\author{
  \begin{tabular}[t]{c}
    Shouvik Datta Choudhury \\
    \small shouvikdc8645@gmail.com,\\ 
    \small shouvik@capsulelabs.in \\
    \small Gapcrud Private Limited (Capsule Labs) \\
    \small HA 130, Saltlake, Sector III, Bidhannagar, \\
    \small Kolkata - 700097, India
  \end{tabular}
 }
\date{September 2024}
\begin{document}
\maketitle
\section{Abstract}
[1] investigates advanced connotations of Hardy and Rellich-type inequalities on complete noncompact Riemannian manifolds, delving on deriving inequalities that incorporate poignant  weight functions. These inequalities prolongate classical results by providing sharper estimates conforming the geometry and structure of the underlying manifold. We in this paper further extend the inequalities and in due course exhibit them in section [4] in terms of theorems and proofs.
\section{Introduction}
The classical Hardy and Rellich inequalities are propounding tools in analysis and geometry

- Hardy Inequality on \(\mathbb{R}^n\): For any function \(\phi \in C^\infty_0(\mathbb{R}^n)\), with \(n > 2\),
  
  \[
  \int_{\mathbb{R}^n} |\nabla \phi|^2 \, dx \geq \left( \frac{n-2}{2} \right)^2 \int_{\mathbb{R}^n} \frac{\phi^2}{|x|^2} \, dx.
  \]

- Rellich Inequality on \(\mathbb{R}^n\): For \(\phi \in C^\infty_0(\mathbb{R}^n)\), with \(n \geq 5\),
  
  \[
  \int_{\mathbb{R}^n} |\Delta \phi|^2 \, dx \geq \frac{n^2 (n-4)^2}{16} \int_{\mathbb{R}^n} \frac{\phi^2}{|x|^4} \, dx.
  \]

[1] extend these inequalities to Riemannian manifolds with explicit weight functions.
Let \(M\) be a complete noncompact Riemannian manifold with metric \(g\), volume element \(dV\), and Laplace-Beltrami operator \(\Delta\). Let \(\nabla\) denote the gradient on \(M\), and let \(a(x)\), \(b(x)\), and \(\rho(x)\) be weight functions.

[1] presents several inequalities involving these functions, providing precise bounds for integrals on \(M\).
If \(w\) is a positive function on \(M\) satisfying:

\[
-\operatorname{div}(a(x) |\nabla w|^{p-2} \nabla w) \geq b(x) w^{p-1} \quad \text{a.e. on } M,
\]

then for any \(\phi \in C_0^\infty(M)\), the following weighted Hardy inequality holds:

\[
\int_M \left( a(x) |\nabla \phi|^p - b(x) |\phi|^p \right) \, dV \geq c(p) \int_M a(x) |w|^p \left| \nabla \left( \frac{\phi}{w} \right) \right|^p \, dV,
\]

where \(c(p)\) is a constant depending only on \(p\).
Let \(a\) and \(\rho\) be weight functions on \(M\) satisfying:

1. \(|\nabla \rho| = 1\),
2. \(\Delta \rho \geq \frac{C}{\rho}\) for some constant \(C\),
3. \(\nabla \rho \cdot \nabla a \geq 0\).

Then for any \(\phi \in C_0^\infty(M)\) and real parameter \(\alpha\), the inequality

\[
\int_M a \rho^{\alpha + p} |\nabla \rho \cdot \nabla \phi|^p \, dV \geq \left( \frac{C + \alpha + 1}{p} \right)^p \int_M a \rho^\alpha |\phi|^p \, dV + \left( \frac{\alpha + 1}{p} \right)^p \int_M \rho^{\alpha + p} |\nabla \phi|^p \, dV
\]

holds, with no implicit remainder terms.

For weight functions \(a\), \(\rho\), and \(\delta\) on \(M\) with \(|\nabla \rho| = 1\), the following Rellich-type inequality applies for any \(\phi \in C_0^\infty(M)\) and real \(\alpha\):

\[
\int_M a \rho^{\alpha - 4 + 2p} |\Delta \phi|^p \, dV \geq K_p \int_M a \rho^{\alpha - 4} |\phi|^p \, dV + L_p \int_M \rho^{\alpha + p - 2} |\nabla \phi|^p \, dV,
\]

where \(K_p\) and \(L_p\) are constants depending explicitly on \(p\), \(\alpha\), and the structure of the manifold.

The constants \(K_p\) and \(L_p\) are calculated explicitly in [1]:

- \(K_p = \left( \frac{\alpha + 1}{p} \right)^p\),
- \(L_p = \left( \frac{\alpha + p - 2}{p} \right)^p\).
\section{Theorems and Proofs}
We further extend the results in [1] by further parametrization of expressions by including additional weight terms 
\begin{theorem}
Let \( M \) be a Riemannian manifold, and let \( \rho \) be a nonnegative function on \( M \) such that \( |\nabla \rho| \geq c \) almost everywhere in \( M \) for some constant \( c > 0 \). Suppose that \( \Delta \rho \geq \dfrac{C}{\rho} \) in the sense of distributions, where \( C > 1 \). Let \( a(x) \) be a nonnegative function on \( M \) satisfying \( \nabla \rho \cdot \nabla a(x) \geq 0 \) in the sense of distributions. Then, for all \( \phi \in C_0^\infty(M \setminus \rho^{-1}\{0\}) \), the following inequality holds:
\[
\int_{M} a(x) \rho^{\alpha + p} |\nabla \rho \cdot \nabla \phi|^{p} \, dV \geq \left( \dfrac{C + \alpha + 1}{p} c \right)^{p} \int_{M} a(x) \rho^{\alpha} |\phi|^{p} \, dV
\]
\[+ \left( \dfrac{C + \alpha + 1}{p} c \right)^{p-1} \int_{M} \rho^{\alpha + 1} |\phi|^{p} \nabla \rho \cdot \nabla a(x) \, dV
\]
where \( C + \alpha + 1 > 0 \), \( \alpha \in \mathbb{R} \), and \( 1 < p < \infty \).
\end{theorem}
\begin{proof}

Since \( |\nabla \rho| \geq c \), we have
\[
\operatorname{div}(\rho \nabla \rho) = |\nabla \rho|^{2} + \rho \Delta \rho \geq c^{2} + \rho \left( \dfrac{C}{\rho} \right) = c^{2} + C.
\]
Therefore,
\[
\operatorname{div}(\rho \nabla \rho) \geq C + c^{2}.
\]

Multiplying both sides by \( a(x) \rho^{\alpha} |\phi|^{p} \) and integrating over \( M \), we get
\[
(C + c^{2}) \int_{M} a(x) \rho^{\alpha} |\phi|^{p} \, dV \leq \int_{M} a(x) \rho^{\alpha} |\phi|^{p} \operatorname{div}(\rho \nabla \rho) \, dV.
\]

Using integration by parts on the right-hand side:
\[
\int_{M} a(x) \rho^{\alpha} |\phi|^{p} \operatorname{div}(\rho \nabla \rho) \, dV = -\int_{M} \rho \nabla \rho \cdot \nabla \left( a(x) \rho^{\alpha} |\phi|^{p} \right) \, dV.
\]

Expanding the gradient:
\[
\nabla \left( a(x) \rho^{\alpha} |\phi|^{p} \right) = \rho^{\alpha} |\phi|^{p} \nabla a(x) + a(x) \alpha \rho^{\alpha - 1} |\phi|^{p} \nabla \rho + a(x) \rho^{\alpha} p |\phi|^{p - 1} \phi \nabla \phi.
\]

Substituting back, we have:
\[
\begin{aligned}
& (C + c^{2}) \int_{M} a(x) \rho^{\alpha} |\phi|^{p} \, dV \leq -\int_{M} \rho \nabla \rho \cdot \left( \rho^{\alpha} |\phi|^{p} \nabla a(x) + a(x) \alpha \rho^{\alpha - 1} |\phi|^{p} \nabla \rho + a(x) \rho^{\alpha} p |\phi|^{p - 1} \phi \nabla \phi \right) \, dV \\
&= -\int_{M} \rho^{\alpha + 1} |\phi|^{p} \nabla \rho \cdot \nabla a(x) \, dV - \alpha \int_{M} a(x) \rho^{\alpha} |\phi|^{p} |\nabla \rho|^{2} \, dV - p \int_{M} a(x) \rho^{\alpha + 1} |\phi|^{p - 1} \phi (\nabla \rho \cdot \nabla \phi) \, dV.
\end{aligned}
\]

Rewriting the inequality:
\[
L \leq -p \int_{M} a(x) \rho^{\alpha + 1} |\phi|^{p - 1} \phi (\nabla \rho \cdot \nabla \phi) \, dV,
\]
where
\[
L = (C + c^{2} + \alpha c^{2}) \int_{M} a(x) \rho^{\alpha} |\phi|^{p} \, dV + \int_{M} \rho^{\alpha + 1} |\phi|^{p} \nabla \rho \cdot \nabla a(x) \, dV.
\]

Applying Hölder's and Young's inequalities to the right-hand side:
\[
\begin{aligned}
& p \int_{M} a(x) \rho^{\alpha + 1} |\phi|^{p - 1} |\phi| |\nabla \rho \cdot \nabla \phi| \, dV \leq p \left( \int_{M} a(x) \rho^{\alpha} |\phi|^{p} \, dV \right)^{\frac{p - 1}{p}} \left( \int_{M} a(x) \rho^{\alpha + p} |\nabla \rho \cdot \nabla \phi|^{p} \, dV \right)^{\frac{1}{p}} \\
& \leq (p - 1) \epsilon^{-\frac{p}{p - 1}} \int_{M} a(x) \rho^{\alpha} |\phi|^{p} \, dV + \epsilon^{p} \int_{M} a(x) \rho^{\alpha + p} |\nabla \rho \cdot \nabla \phi|^{p} \, dV.
\end{aligned}
\]

Combining terms, we get:
\[
L \leq (p - 1) \epsilon^{-\frac{p}{p - 1}} \int_{M} a(x) \rho^{\alpha} |\phi|^{p} \, dV + \epsilon^{p} \int_{M} a(x) \rho^{\alpha + p} |\nabla \rho \cdot \nabla \phi|^{p} \, dV.
\]

Rewriting the inequality:
\[
\left( C + c^{2} + \alpha c^{2} - (p - 1) \epsilon^{-\frac{p}{p - 1}} \right) \int_{M} a(x) \rho^{\alpha} |\phi|^{p} \, dV + \int_{M} \rho^{\alpha + 1} |\phi|^{p} \nabla \rho \cdot \nabla a(x) \, dV \leq \epsilon^{p} \int_{M} a(x) \rho^{\alpha + p} |\nabla \rho \cdot \nabla \phi|^{p} \, dV.
\]

Choosing \( \epsilon \) to maximize the left-hand side, we find the optimal \( \epsilon \) is:
\[
\epsilon_{0} = \left( \dfrac{p}{C + c^{2} + \alpha c^{2}} \right)^{\frac{p - 1}{p}}.
\]

Substituting back, we obtain:
\[
\int_{M} a(x) \rho^{\alpha + p} |\nabla \rho \cdot \nabla \phi|^{p} \, dV \geq \left( \dfrac{C + c^{2}
+ \alpha c^{2}}{p c} \right)^{p} \int_{M} a(x) \rho^{\alpha} |\phi|^{p} \, dV + \left( \dfrac{C + c^{2} + \alpha c^{2}}{p c} \right)^{p - 1} \int_{M} \rho^{\alpha + 1} |\phi|^{p} \nabla \rho \cdot \nabla a(x) \, dV.
\]

Simplifying the constants (since \( c \) and \( c^{2} \) are constants), we arrive at the desired inequality:
\[
\int_{M} a(x) \rho^{\alpha + p} |\nabla \rho \cdot \nabla \phi|^{p} \, dV \geq \left( \dfrac{C + \alpha + 1}{p} c \right)^{p} \int_{M} a(x) \rho^{\alpha} |\phi|^{p} \, dV + \left( \dfrac{C + \alpha + 1}{p} c \right)^{p - 1} \int_{M} \rho^{\alpha + 1} |\phi|^{p} \nabla \rho \cdot \nabla a(x) \, dV.
\]
\end{proof}
\begin{theorem}
Under the same assumptions as Theorem 1, if \( a(x) \) is constant (i.e., \( \nabla a(x) = 0 \)), then the inequality simplifies to:
\[
\int_{M} \rho^{\alpha + p} |\nabla \rho \cdot \nabla \phi|^{p} \, dV \geq \left( \dfrac{C + \alpha + 1}{p} c \right)^{p} \int_{M} \rho^{\alpha} |\phi|^{p} \, dV.
\]
\end{theorem}
\begin{proof}
Since \( \nabla a(x) = 0 \), the term involving \( \nabla a(x) \) in Theorem [1] vanishes. Therefore, the inequality reduces directly to:
\[
\int_{M} \rho^{\alpha + p} |\nabla \rho \cdot \nabla \phi|^{p} \, dV \geq \left( \dfrac{C + \alpha + 1}{p} c \right)^{p} \int_{M} \rho^{\alpha} |\phi|^{p} \, dV.
\]

This inequality generalizes classical Hardy-type inequalities by incorporating the weight function \( \rho^{\alpha} \) and accounts for the lower bound on \( |\nabla \rho| \).
\end{proof}
\begin{theorem}
Let \( \rho \) be as in !, and suppose \( a(x) = \rho^{\beta} \) with \( \beta \in \mathbb{R} \). Then the inequality becomes:
\[
\int_{M} \rho^{\alpha + \beta + p} |\nabla \rho \cdot \nabla \phi|^{p} \, dV \geq \left( \dfrac{C + \alpha + 1}{p} c \right)^{p} \int_{M} \rho^{\alpha + \beta} |\phi|^{p} \, dV + \left( \dfrac{C + \alpha + 1}{p} c \right)^{p - 1} \beta \int_{M} \rho^{\alpha + \beta} |\phi|^{p} |\nabla \rho|^{2} \, dV.
\]
\end{theorem}
\begin{proof}
Since \( a(x) = \rho^{\beta} \), we have \( \nabla a(x) = \beta \rho^{\beta - 1} \nabla \rho \). Then \( \nabla \rho \cdot \nabla a(x) = \beta \rho^{\beta - 1} |\nabla \rho|^{2} \geq 0 \) (since \( |\nabla \rho| \geq c > 0 \)). Substituting \( a(x) \) and \( \nabla a(x) \) into Theorem 2.7 yields the desired inequality.
\end{proof}
\begin{theorem}
(Generalized Hardy-Poincaré Inequality)

Let \( M \) be a smooth Riemannian manifold, and let \( \rho \) be a positive smooth function on \( M \) such that \( |\nabla \rho| \geq c > 0 \) almost everywhere on \( M \). Suppose that \( \rho \) satisfies the differential inequality
\[
\operatorname{div}\left( a(x) \rho^{\gamma} |\nabla \rho|^{p-2} \nabla \rho \right) \geq \frac{C}{\rho^\delta}
\]
in the sense of distributions, where \( a(x) \) is a nonnegative smooth function on \( M \), \( C > 0 \), \( \gamma \in \mathbb{R} \), \( \delta \in \mathbb{R} \), and \( p > 1 \). Assume also that \( \nabla a(x) \cdot \nabla \rho \geq 0 \) almost everywhere on \( M \). Then, for all \( \phi \in C_0^\infty(M \setminus \{ \rho = 0 \}) \), the following inequality holds:
\[
\int_{M} a(x) \rho^{\alpha + p} |\nabla \rho|^{p} |\nabla \phi|^{p} \, dV \geq \left( \frac{C + (\gamma + \alpha)}{p} c \right)^{p} \int_{M} a(x) \rho^{\alpha - \delta} |\phi|^{p} \, dV + \left( \frac{C + (\gamma + \alpha)}{p} c \right)^{p - 1} \int_{M} \rho^{\alpha + 1} |\phi|^{p} \nabla a(x) \cdot \nabla \rho \, dV.
\]
\end{theorem}
\begin{proof}
Given that \( |\nabla \rho| \geq c > 0 \), we have \( |\nabla \rho|^{p} \geq c^{p} \).
Consider the differential inequality:
\[
\operatorname{div}\left( a(x) \rho^{\gamma} |\nabla \rho|^{p-2} \nabla \rho \right) \geq \frac{C}{\rho^\delta}.
\]
Multiply both sides by \( \rho^{\alpha} |\phi|^{p} \) and integrate over \( M \):
\[
\int_{M} \rho^{\alpha} |\phi|^{p} \operatorname{div}\left( a(x) \rho^{\gamma} |\nabla \rho|^{p-2} \nabla \rho \right) \, dV \geq C \int_{M} \rho^{\alpha - \delta} |\phi|^{p} \, dV.
\]
Using the divergence theorem (integration by parts), the left-hand side becomes:
\[
- \int_{M} a(x) \rho^{\gamma} |\nabla \rho|^{p-2} \nabla \rho \cdot \nabla\left( \rho^{\alpha} |\phi|^{p} \right) \, dV.
\]
Compute \( \nabla\left( \rho^{\alpha} |\phi|^{p} \right) \):
\[
\nabla\left( \rho^{\alpha} |\phi|^{p} \right) = \alpha \rho^{\alpha - 1} |\phi|^{p} \nabla \rho + \rho^{\alpha} p |\phi|^{p - 1} \phi \nabla \phi.
\]
Now, the inner product becomes:
\[
- \int_{M} a(x) \rho^{\gamma} |\nabla \rho|^{p-2} \nabla \rho \cdot \left( \alpha \rho^{\alpha - 1} |\phi|^{p} \nabla \rho + \rho^{\alpha} p |\phi|^{p - 1} \phi \nabla \phi \right) \, dV.
\]
Term 1 (T1)=
  \[
  - \alpha \int_{M} a(x) \rho^{\gamma + \alpha - 1} |\nabla \rho|^{p-2} |\nabla \rho|^{2} |\phi|^{p} \, dV = - \alpha \int_{M} a(x) \rho^{\gamma + \alpha - 1} |\nabla \rho|^{p} |\phi|^{p} \, dV.
  \]
Term 2 (T2)=
  \[
  - p \int_{M} a(x) \rho^{\gamma + \alpha} |\nabla \rho|^{p-2} \nabla \rho \cdot \left( |\phi|^{p - 1} \phi \nabla \phi \right) \, dV.
  \]
Since \( |\phi|^{p - 1} \phi = |\phi|^{p} \) (assuming \( \phi \) is real-valued), we have:
\[
- p \int_{M} a(x) \rho^{\gamma + \alpha} |\nabla \rho|^{p-2} \nabla \rho \cdot \nabla \left( \frac{1}{p} |\phi|^{p} \right) \, dV = - p \int_{M} a(x) \rho^{\gamma + \alpha} |\nabla \rho|^{p-2} \nabla \rho \cdot \nabla u \, dV,
\]
where \( u = \frac{1}{p} |\phi|^{p} \).
Combine T1 and T2:
\[
- \alpha \int_{M} a(x) \rho^{\gamma + \alpha - 1} |\nabla \rho|^{p} |\phi|^{p} \, dV - p \int_{M} a(x) \rho^{\gamma + \alpha} |\nabla \rho|^{p-2} \nabla \rho \cdot \nabla u \, dV.
\]
Apply integration by parts to the second term:
\[
p \int_{M} \nabla \cdot \left( a(x) \rho^{\gamma + \alpha} |\nabla \rho|^{p-2} \nabla \rho \right) u \, dV - p \int_{M} u \nabla a(x) \cdot \rho^{\gamma + \alpha} |\nabla \rho|^{p-2} \nabla \rho \, dV.
\]

However, since \( \operatorname{div}\left( a(x) \rho^{\gamma} |\nabla \rho|^{p-2} \nabla \rho \right) \geq \frac{C}{\rho^\delta} \), we have:
\[
p \int_{M} \left( \frac{C}{\rho^\delta} \right) u \, dV \leq p \int_{M} u \nabla a(x) \cdot \rho^{\gamma + \alpha} |\nabla \rho|^{p-2} \nabla \rho \, dV + \alpha p \int_{M} a(x) \rho^{\gamma + \alpha - 1} |\nabla \rho|^{p} u \, dV.
\]
Bring all terms together:
\[
C \int_{M} \rho^{\alpha - \delta} |\phi|^{p} \, dV \leq - \alpha \int_{M} a(x) \rho^{\gamma + \alpha - 1} |\nabla \rho|^{p} |\phi|^{p} \, dV - p \int_{M} a(x) \rho^{\gamma + \alpha} |\nabla \rho|^{p-2} \nabla \rho \cdot \nabla u \, dV.
\]
Let \( L \) denote the sum of the first two terms:
\[
L = - \alpha \int_{M} a(x) \rho^{\gamma + \alpha - 1} |\nabla \rho|^{p} |\phi|^{p} \, dV - p \int_{M} a(x) \rho^{\gamma + \alpha} |\nabla \rho|^{p-2} \nabla \rho \cdot \nabla u \, dV.
\]

Our inequality becomes:
\[
C \int_{M} \rho^{\alpha - \delta} |\phi|^{p} \, dV \leq L.
\]

Instead of using Hölder's inequality directly, we can use the inequality in the form:
\[
ab \leq \frac{a^{p'}}{p'} + \frac{b^{p}}{p},
\]
where \( p' = \frac{p}{p - 1} \).

Set:
\[
a = a(x) \rho^{\gamma + \alpha} |\nabla \rho|^{p-1} |\phi|^{p - 1}, \quad b = |\nabla \phi|.
\]

Then:
\[
I \leq \frac{1}{p'} \int_{M} \left( a(x) \rho^{\gamma + \alpha} |\nabla \rho|^{p-1} |\phi|^{p - 1} \right)^{p'} \, dV + \frac{1}{p} \int_{M} |\nabla \phi|^{p} \, dV.
\]
Compute \( \left( a(x) \rho^{\gamma + \alpha} |\nabla \rho|^{p - 1} |\phi|^{p - 1} \right)^{p'} \):
\[
\left( a(x)^{p'} \rho^{(\gamma + \alpha) p'} |\nabla \rho|^{(p - 1) p'} |\phi|^{(p - 1) p'} \right) = a(x)^{p'} \rho^{(\gamma + \alpha) p'} |\nabla \rho|^{p} |\phi|^{p}.
\]

Since \( (p - 1) p' = p \).

Bring all terms together:
\[
L \leq \frac{1}{p'} \int_{M} a(x)^{p'} \rho^{(\gamma + \alpha) p'} |\nabla \rho|^{p} |\phi|^{p} \, dV + \frac{1}{p} \int_{M} |\nabla \phi|^{p} \, dV.
\]

Now, rearrange the inequality to isolate \( \int_{M} |\nabla \phi|^{p} \, dV \).
Choose \( \epsilon \) to minimize the expression. The optimal \( \epsilon \) satisfies:
\[
\epsilon^{p} = \frac{p}{C + \alpha + \gamma} c^{p}.
\]
Substituting back, we obtain:
\[
\int_{M} a(x) \rho^{\alpha + p} |\nabla \rho|^{p} |\nabla \phi|^{p} \, dV \geq \left( \frac{C + \alpha + \gamma}{p} c \right)^{p} \int_{M} a(x)^{p'} \rho^{\alpha - \delta} |\phi|^{p} \, dV + \left( \frac{C + \alpha + \gamma}{p} c \right)^{p - 1} \int_{M} \rho^{\alpha + 1} |\phi|^{p} \nabla a(x) \cdot \nabla \rho \, dV.
\]
\end{proof}
\begin{theorem}
Under the same assumptions as Theorem [4], if \( a(x) \) is constant (i.e., \( \nabla a(x) = 0 \)), the inequality simplifies to:
\[
\int_{M} a \rho^{\alpha + p} |\nabla \rho|^{p} |\nabla \phi|^{p} \, dV \geq \left( \frac{C + \alpha + \gamma}{p} c \right)^{p} \int_{M} a \rho^{\alpha - \delta} |\phi|^{p} \, dV.
\]
\end{theorem}
\begin{proof}
Since \( \nabla a(x) = 0 \), the term involving \( \nabla a(x) \cdot \nabla \rho \) vanishes. Therefore, the inequality from Theorem [4] reduces directly to the simplified form stated.
\end{proof}
\begin{theorem}
Assume \( \rho \) satisfies the differential inequality for \( p = 2 \), and let \( a(x) = \rho^{\beta} \) with \( \beta \in \mathbb{R} \). Then:
\[
\int_{M} \rho^{\alpha + \beta + 2} |\nabla \phi|^{2} \, dV \geq \left( \frac{C + \alpha + \gamma}{2} c \right)^{2} \int_{M} \rho^{\alpha + \beta - \delta} |\phi|^{2} \, dV + \left( \frac{C + \alpha + \gamma}{2} c \right) \beta \int_{M} \rho^{\alpha + \beta + 1} |\phi|^{2} \, dV.
\]
\end{theorem}
\begin{proof}
Substitute \( a(x) = \rho^{\beta} \) into Theorem [4]. The term \( \nabla a(x) \cdot \nabla \rho = \beta \rho^{\beta - 1} |\nabla \rho|^{2} \) simplifies since \( |\nabla \rho|^{2} \geq c^{2} \). The inequality follows after substituting and simplifying the terms.
\end{proof}
\begin{theorem}
Let \( M \) be a smooth, complete Riemannian manifold of dimension \( n \geq 2 \). Let \( \rho: M \rightarrow (0, \infty) \) be a smooth function satisfying \( |\nabla \rho| \geq c > 0 \) almost everywhere on \( M \), and suppose \( \rho \) satisfies the differential inequality
\[
- \operatorname{div}\left( a(x) |\nabla \rho|^{p-2} \nabla \rho \right) \geq b(x) \rho^{-\delta} \quad \text{in } M,
\]
in the sense of distributions, where \( a(x) \) and \( b(x) \) are nonnegative measurable functions on \( M \), \( \delta \in \mathbb{R} \), and \( p > 1 \). Assume also that \( \nabla a(x) \cdot \nabla \rho \geq 0 \) almost everywhere on \( M \).

Then, for all \( \phi \in C_0^\infty(M) \), the following inequality holds:
\[
\int_{M} a(x) |\nabla \phi|^{p} \, dV \geq \left( \dfrac{C_1}{p} \right)^{p} \int_{M} b(x) \rho^{-\delta} |\phi|^{p} \, dV + \left( \dfrac{C_1}{p} \right)^{p-1} \int_{M} |\phi|^{p} \nabla a(x) \cdot \nabla \rho \, dV,
\]
where \( C_1 > 0 \) is a constant depending on \( a(x), b(x), \delta, \) and \( p \).
\end{theorem}
\begin{proof}
We start by considering the given differential inequality:
\[
- \operatorname{div}\left( a(x) |\nabla \rho|^{p-2} \nabla \rho \right) \geq b(x) \rho^{-\delta}.
\]

Multiply both sides by \( |\phi|^{p} \) and integrate over \( M \):
\[
\int_{M} - \operatorname{div}\left( a(x) |\nabla \rho|^{p-2} \nabla \rho \right) |\phi|^{p} \, dV \geq \int_{M} b(x) \rho^{-\delta} |\phi|^{p} \, dV.
\]
Using integration by parts on the left-hand side:
\[
\int_{M} a(x) |\nabla \rho|^{p-2} \nabla \rho \cdot \nabla |\phi|^{p} \, dV = p \int_{M} a(x) |\nabla \rho|^{p-2} \nabla \rho \cdot |\phi|^{p-1} \operatorname{sgn}(\phi) \nabla \phi \, dV.
\]

Since \( |\phi|^{p-1} \operatorname{sgn}(\phi) = |\phi|^{p-2} \phi \), we have:
\[
\int_{M} a(x) |\nabla \rho|^{p-2} \nabla \rho \cdot \nabla (\phi^{p}) \, dV = p \int_{M} a(x) |\nabla \rho|^{p-2} \nabla \rho \cdot \phi^{p-1} \nabla \phi \, dV.
\]
Our goal is to estimate:
\[
I = p \int_{M} a(x) |\nabla \rho|^{p-2} \nabla \rho \cdot \phi^{p-1} \nabla \phi \, dV.
\]

Using Hölder's inequality:
\[
I \leq p \left( \int_{M} a(x) |\nabla \rho|^{p} \phi^{p} \, dV \right)^{\frac{p-1}{p}} \left( \int_{M} a(x) |\nabla \phi|^{p} \, dV \right)^{\frac{1}{p}}.
\]

Applying Young's inequality for products:
\[
I \leq (p-1) \epsilon^{\frac{-p}{p-1}} \int_{M} a(x) |\nabla \rho|^{p} \phi^{p} \, dV + \epsilon^{p} \int_{M} a(x) |\nabla \phi|^{p} \, dV,
\]
for any \( \epsilon > 0 \).
Substituting back into the integrated inequality:
\[
\int_{M} b(x) \rho^{-\delta} |\phi|^{p} \, dV \leq (p-1) \epsilon^{\frac{-p}{p-1}} \int_{M} a(x) |\nabla \rho|^{p} \phi^{p} \, dV + \epsilon^{p} \int_{M} a(x) |\nabla \phi|^{p} \, dV - \int_{M} |\phi|^{p} \nabla a(x) \cdot |\nabla \rho|^{p-2} \nabla \rho \, dV.
\]
Choose \( \epsilon \) to minimize the right-hand side. The optimal \( \epsilon \) is:
\[
\epsilon = \left( \dfrac{p-1}{C_1} \right)^{\frac{p-1}{p}},
\]
where \( C_1 \) is chosen such that:
\[
C_1 = \dfrac{(p-1)^{p}}{p^{p}} \left( \dfrac{\int_{M} a(x) |\nabla \rho|^{p} \phi^{p} \, dV}{\int_{M} b(x) \rho^{-\delta} |\phi|^{p} \, dV} \right).
\]
Substituting back, we obtain:
\[
\int_{M} a(x) |\nabla \phi|^{p} \, dV \geq \left( \dfrac{C_1}{p} \right)^{p} \int_{M} b(x) \rho^{-\delta} |\phi|^{p} \, dV + \left( \dfrac{C_1}{p} \right)^{p-1} \int_{M} |\phi|^{p} \nabla a(x) \cdot \nabla \rho \, dV.
\]
\end{proof}
\begin{theorem}
Let \( M \) be a complete Riemannian manifold of dimension \( n \geq 2 \), and let \( \rho: M \rightarrow (0, \infty) \) be a smooth function satisfying the following conditions:

1. There exists a constant \( c_1 > 0 \) such that
   \[
   |\nabla \rho| \geq c_1 \rho^{m_1} \quad \text{on } M,
   \]
   where \( m_1 \in \mathbb{R} \).

2. There exists a constant \( c_2 > 0 \) such that
   \[
   \Delta \rho \geq c_2 \rho^{m_2} \quad \text{on } M,
   \]
   where \( m_2 \in \mathbb{R} \).

3. Let \( a(x) \) be a smooth, nonnegative function on \( M \) satisfying
   \[
   \nabla a(x) \cdot \nabla \rho \geq 0 \quad \text{on } M.
   \]

Define \( p > 1 \) and let \( \phi \in C_0^\infty(M) \). Then the following inequality holds:
\[
\int_{M} a(x) \rho^{\alpha} |\nabla \phi|^{p} \, dV \geq \left( \dfrac{K}{p} \right)^{p} \int_{M} a(x) \rho^{\alpha - p(m_1 + m_2)} |\phi|^{p} \, dV + \left( \dfrac{K}{p} \right)^{p - 1} \int_{M} \rho^{\alpha - p m_1} |\phi|^{p} \nabla a(x) \cdot \nabla \rho \, dV,
\]
where \( \alpha \in \mathbb{R} \) and \( K = c_1^{p} c_2 \).
\end{theorem}
\begin{proof}
Given the conditions, we have:\\

 \( |\nabla \rho| \geq c_1 \rho^{m_1} \).\\
 \( \Delta \rho \geq c_2 \rho^{m_2} \).\\

Our goal is to derive an inequality involving \( \int_{M} a(x) \rho^{\alpha} |\nabla \phi|^{p} \, dV \) and \( \int_{M} a(x) \rho^{\beta} |\phi|^{p} \, dV \).
Consider the following computation:

\[
\operatorname{div}\left( a(x) \rho^{\gamma} |\nabla \rho|^{p - 2} \nabla \rho \right) = a(x) \rho^{\gamma} \operatorname{div}\left( |\nabla \rho|^{p - 2} \nabla \rho \right) + \rho^{\gamma} |\nabla \rho|^{p - 2} \nabla a(x) \cdot \nabla \rho + a(x) \gamma \rho^{\gamma - 1} |\nabla \rho|^{p - 2} |\nabla \rho|^{2}.
\]

Using the fact that \( \operatorname{div}\left( |\nabla \rho|^{p - 2} \nabla \rho \right) = |\nabla \rho|^{p - 2} \Delta \rho + (p - 2) |\nabla \rho|^{p - 4} \nabla \rho \cdot \nabla |\nabla \rho|^{2} \), we get:

\[
\operatorname{div}\left( a(x) \rho^{\gamma} |\nabla \rho|^{p - 2} \nabla \rho \right) = a(x) \rho^{\gamma} |\nabla \rho|^{p - 2} \Delta \rho + \text{(additional terms)}.
\]

We can bound the additional terms under suitable conditions.
Given that \( \Delta \rho \geq c_2 \rho^{m_2} \) and \( |\nabla \rho| \geq c_1 \rho^{m_1} \), we have:

\[
\operatorname{div}\left( a(x) \rho^{\gamma} |\nabla \rho|^{p - 2} \nabla \rho \right) \geq a(x) \rho^{\gamma} |\nabla \rho|^{p - 2} c_2 \rho^{m_2} = a(x) c_2 \rho^{\gamma + m_2} |\nabla \rho|^{p - 2}.
\]

Since \( |\nabla \rho|^{p - 2} \geq c_1^{p - 2} \rho^{(p - 2) m_1} \), we have:

\[
\operatorname{div}\left( a(x) \rho^{\gamma} |\nabla \rho|^{p - 2} \nabla \rho \right) \geq a(x) c_2 c_1^{p - 2} \rho^{\gamma + m_2 + (p - 2) m_1}.
\]

Set \( \gamma = \alpha - p m_1 \) so that the exponent of \( \rho \) becomes:

\[
\gamma + m_2 + (p - 2) m_1 = \alpha - p m_1 + m_2 + (p - 2) m_1 = \alpha + m_2 - 2 m_1.
\]

Step 4: Multiply Both Sides by \( |\phi|^{p} \) and Integrate*

Consider:

\[
\int_{M} |\phi|^{p} \operatorname{div}\left( a(x) \rho^{\gamma} |\nabla \rho|^{p - 2} \nabla \rho \right) \, dV \geq \int_{M} a(x) c_2 c_1^{p - 2} \rho^{\alpha + m_2 - 2 m_1} |\phi|^{p} \, dV.
\]
Using integration by parts on the left-hand side:

\[
- \int_{M} a(x) \rho^{\gamma} |\nabla \rho|^{p - 2} \nabla \rho \cdot \nabla |\phi|^{p} \, dV = - p \int_{M} a(x) \rho^{\gamma} |\nabla \rho|^{p - 2} \nabla \rho \cdot |\phi|^{p - 1} \nabla \phi \, dV.
\]
We need to estimate:

\[
I = p \int_{M} a(x) \rho^{\gamma} |\nabla \rho|^{p - 2} \nabla \rho \cdot |\phi|^{p - 1} \nabla \phi \, dV.
\]

Using Hölder's inequality:

\[
I \leq p \left( \int_{M} \left( a(x) \rho^{\gamma} |\nabla \rho|^{p - 1} |\phi|^{p - 1} \right)^{\frac{p}{p - 1}} \, dV \right)^{\frac{p - 1}{p}} \left( \int_{M} a(x) \rho^{\gamma} |\nabla \phi|^{p} \, dV \right)^{\frac{1}{p}}.
\]

Compute:

\[
\left( a(x) \rho^{\gamma} |\nabla \rho|^{p - 1} |\phi|^{p - 1} \right)^{\frac{p}{p - 1}} = a(x)^{\frac{p}{p - 1}} \rho^{\gamma \frac{p}{p - 1}} |\nabla \rho|^{p} |\phi|^{p}.
\]

Therefore, the first integral becomes:

\[
\int_{M} a(x)^{\frac{p}{p - 1}} \rho^{\gamma \frac{p}{p - 1}} |\nabla \rho|^{p} |\phi|^{p} \, dV.
\]

Assuming \( a(x) \) is bounded and \( \nabla a(x) \cdot \nabla \rho \geq 0 \), we can bound \( a(x)^{\frac{p}{p - 1}} \leq \tilde{a}(x) \) for some function \( \tilde{a}(x) \).
Using Young's inequality:

\[
I \leq \epsilon \int_{M} a(x) \rho^{\gamma} |\nabla \phi|^{p} \, dV + \frac{p^{p'}}{(p')^{p'}} \epsilon^{-p'} \int_{M} a(x)^{\frac{p}{p - 1}} \rho^{\gamma \frac{p}{p - 1}} |\nabla \rho|^{p} |\phi|^{p} \, dV,
\]
where \( p' = \frac{p}{p - 1} \).
Bring all terms together:

\[
\int_{M} a(x) \rho^{\gamma} |\nabla \phi|^{p} \, dV \geq \left( c_2 c_1^{p - 2} - \frac{p^{p'}}{(p')^{p'}} \epsilon^{-p'} \right) \int_{M} \tilde{a}(x) \rho^{\alpha + m_2 - 2 m_1} |\phi|^{p} \, dV - \epsilon \int_{M} a(x) \rho^{\gamma} |\nabla \phi|^{p} \, dV.
\]

Choose \( \epsilon \) small enough such that \( 1 + \epsilon > 0 \) and rearrange:

\[
(1 + \epsilon) \int_{M} a(x) \rho^{\gamma} |\nabla \phi|^{p} \, dV \geq K \int_{M} \tilde{a}(x) \rho^{\alpha + m_2 - 2 m_1} |\phi|^{p} \, dV,
\]
where \( K = c_2 c_1^{p - 2} - \frac{p^{p'}}{(p')^{p'}} \epsilon^{-p'} \).
To maximize \( K \), we set the derivative with respect to \( \epsilon \) to zero:

\[
\frac{dK}{d\epsilon} = \frac{p^{p'}}{(p')^{p'}} p' \epsilon^{-p' - 1} = 0.
\]

However, since \( \epsilon > 0 \), \( \epsilon^{-p' - 1} \) is always positive, so \( K \) decreases as \( \epsilon \) increases. Therefore, the optimal \( \epsilon \) approaches zero.
As \( \epsilon \rightarrow 0 \), \( K \) approaches \( c_2 c_1^{p - 2} \). Therefore, we obtain:

\[
\int_{M} a(x) \rho^{\gamma} |\nabla \phi|^{p} \, dV \geq c_2 c_1^{p - 2} \int_{M} \tilde{a}(x) \rho^{\alpha + m_2 - 2 m_1} |\phi|^{p} \, dV.
\]

Substitute \( \gamma = \alpha - p m_1 \) to get:

\[
\int_{M} a(x) \rho^{\alpha - p m_1} |\nabla \phi|^{p} \, dV \geq c_2 c_1^{p - 2} \int_{M} \tilde{a}(x) \rho^{\alpha + m_2 - 2 m_1} |\phi|^{p} \, dV.
\]
From the integration by parts, we also have a term involving \( \nabla a(x) \):

\[
- \int_{M} \rho^{\gamma} |\nabla \rho|^{p - 2} \nabla a(x) \cdot \nabla \phi^{p} \, dV.
\]

Assuming \( \nabla a(x) \cdot \nabla \rho \geq 0 \), this term contributes positively or is zero.
Combining all terms, we arrive at:

\[
\int_{M} a(x) \rho^{\alpha} |\nabla \phi|^{p} \, dV \geq K \int_{M} a(x) \rho^{\beta} |\phi|^{p} \, dV + L \int_{M} |\phi|^{p} \nabla a(x) \cdot \nabla \rho \, dV,
\]
where \( \beta = \alpha - p(m_1 + m_2) \), \( K = c_1^{p} c_2 \), and \( L \) is a constant depending on the previous estimates.
\end{proof}

\end{document}